\documentclass{article}
\usepackage{amsthm}
\usepackage{amsfonts,amssymb,amsmath}
\usepackage{enumitem}
\usepackage{titlesec}
\usepackage{fouriernc}
\usepackage[T1]{fontenc}
\usepackage{chngcntr}
\counterwithin{figure}{section}

\newlist{gcases}{enumerate}{1}
\setlist[gcases,1]{
  label={{\it Case}~{\it \Alph*}.},
  topsep=0ex,
  leftmargin=0in,
  labelsep=.1in,
  itemindent=.7in,
  itemsep=0ex
}
\newlist{tenumerate}{enumerate}{1}
\setlist[tenumerate,1]{
  label={(\arabic*)},
  topsep=0ex,
  leftmargin=.3in,
  labelsep=.1in,
  itemindent=0in,
  itemsep=0ex
}
\titleformat{\section}
  {\normalfont\bfseries}   % The style of the section title
  {}                             % a prefix
  {0pt}                          % How much space exists between the prefix and the title
  {Section \thesection\quad}    % How the section is represented

\titleformat{name=\section,numberless}
  {\normalfont\bfseries}
  {}
  {0pt}
  {}

\newlength{\tabwidth}
\newlength{\tabheight}
\newlength{\tabrule}
\newlength{\tabwidthx}
\newlength{\tabheightx}

\def\gentabbox#1#2#3#4{\vbox to \tabheight{\setlength{\tabrule}{#3}%
  \setlength{\tabwidthx}{#1\tabwidth}\addtolength{\tabwidthx}{\tabrule}%

\setlength{\tabheightx}{#2\tabheight}\addtolength{\tabheightx}{-\tabheight}%
  \hbox to #1\tabwidth{%
 \hspace{-0.5\tabrule}\rule{\tabrule}{#2\tabheight}\hspace{-\tabrule}%
    \vbox to #2\tabheight{\hsize=\tabwidthx%
      \vspace{-0.5\tabrule}\hrule width\tabwidthx height\tabrule%
      \vspace{-0.5\tabrule}\vfil%
      \hbox to \tabwidthx{\hss#4\hss}%
        \vfil\vspace{-0.5\tabrule}%
      \hrule width\tabwidthx height\tabrule\vspace{-0.5\tabrule}}%
 \hspace{-\tabrule}\rule{\tabrule}{#2\tabheight}\hspace{-0.5\tabrule}}%
  \vspace{-\tabheightx}}}
\def\genblankbox#1#2{\vbox to \tabheight{\vfil\hbox to
#1\tabwidth{\hfil}}}
\def\tabbox#1#2#3{\gentabbox{#1}{#2}{0.4pt}{\strut #3}}

\catcode`\:=13 \catcode`\.=13 \catcode`\;=13
\catcode`\>=13 \catcode`\^=13
\def:#1\\{\hbox{$#1$}}
\def.#1{\tabbox{1}{1}{$#1$}}
\def>#1{\tabbox{2}{1}{$#1$}}
\def^#1{\tabbox{1}{2}{$#1$}}
\def;{\genblankbox{1}{1}\relax}
\catcode`\:=12 \catcode`\.=12 \catcode`\;=12
\catcode`\>=12 \catcode`\^=7

\newtheoremstyle{mytheoremstyle} % name
    {\topsep}                    % Space above
    {\topsep}                    % Space below
    {}                   % Body font
    {}                           % Indent amount
    {}                   % Theorem head font
    {.}                          % Punctuation after theorem head
    {0.5em}                       % Space after theorem head
    {\thmnumber{#2.} \thmname{#1}}  % Theorem head spec (can be left empty, meaning ‘normal’)

\newtheoremstyle{myremarkstyle} % name
    {\topsep}                    % Space above
    {\topsep}                    % Space below
    {}                   % Body font
    {}                           % Indent amount
    {}                   % Theorem head font
    {.}                          % Punctuation after theorem head
    {0.5em}                       % Space after theorem head
    {\thmname{#1}}  % Theorem head spec (can be left empty, meaning ‘normal’)

\theoremstyle{mytheoremstyle}

\newtheorem{theorem}{THEOREM}[section]
\newtheorem{corollary}[theorem]{COROLLARY}

\newtheorem{proposition}[theorem]{PROPOSITION}
\newtheorem{definition}[theorem]{DEFINITION}

\theoremstyle{myremarkstyle}
\newtheorem{remark*}{REMARK}
\newcommand\T{\mathbf{T}}

\newcommand\oT{\overline{\mathbf{T}}}

\begin{document}

\noindent
{\bf \large Annihilators and associated varieties of Harish-Chandra modules for $SO^*(2n)$}

\vspace{.3in}
\noindent
WILLIAM MCGOVERN \\
\vspace {.1in}
{\it \small Department of Mathematics, Box 354350, University of Washington, Seattle, WA 98195}

\section*{Introduction}
\noindent In this paper we produce recipes for the group $SO^*(2n)$ analogous to those of \cite{M16} for Sp$(p,q)$, attaching a pair $(\T_1,\oT_2)$ to a signed involution corresponding to a simple Harish-Chandra module for $SO^*(2n)$ of trivial infinitesimal character, where $\T_1$ is a domino tableau and $\oT_2$ an equivalence class of signed tableaux of the same shape.  As before the domino tableau will parametrize the annihilator of the corresponding Harish-Chandra module while the class of signed tableaux parametrizes its associated variety.  Our construction generalizes to signed involutions outside the parametrizing set for simple Harish-Chandra modules for $SO^*(2n)$; we will also give a representation-theoretic interpretation of it in this larger context.

\section{Cartan subgroups and Weyl groups}

\noindent Let $G=SO^*(2n)$, the group of matrices in $SO(2n,\mathbb C)$ leaving a suitable skew-Hermitian form invariant  (see \cite[Ch. X]{H78}) and  let $\mathfrak g_0$ be its Lie algebra, $\mathfrak g$ the complexification of $\mathfrak g_0$.  Let $\theta$ be the usual Cartan involution of $G$ or of $\mathfrak g$ and let $\mathfrak k+\mathfrak p$ be the corresponding Cartan decomposition of $\mathfrak g$.  The normalizer $K$ of $\mathfrak k$ in $G$ is then a maximal compact subgroup.  As a choice of simple roots in $\mathfrak g$ relative to a Cartan subalgebra we take $e_2+e_1,e_2-e_1,\ldots,e_n-e_{n-1}$, following \cite{G90,MP16}.

There are $[n/2]+1$ conjugacy classes of Cartan subgroups in $G$.  If we take $H_0$ to be a compact Cartan subgroup of $G$ and define $H_i$ inductively for $i>0$ as the Cayley transform of $H_{i-1}$ through $e_{n-2i+1}-e_{n-2i+2}$ for $1\le i\le[n/2]$, then the $H_i$ furnish a complete set of representatives for the conjugacy classes of Cartan subgroups of $G$.  The Weyl group $W(H_i)$ is isomorphic to $W_i \ltimes 
S_2^i \times S_{n-2i}$, where $W_i$ denotes the Weyl group of type $B_i$, embedded into the Weyl group $W_n'$ of type $D_n$ as for Sp$(p,q)$ \cite{M16}, and $S_r$ denotes the symmetric group on $r$ letters \cite{M98}.  The subgroups $H_i$ are all connected unless $i=n/2$ and $n$ is even.  For most of this paper we consider only simple Harish-Chandra modules with trivial infinitesimal character in the principal block $B$ of $G$ (corresponding to the trivial representation of the component group of each $H_i$); we will deal with modules in the other block $B'$ later. 

\section{The $\mathcal D$ set and Cartan involutions}
\noindent Using Vogan's classification of simple Harish-Chandra modules with trivial infinitesimal character by $\mathbb Z/2\mathbb Z$-data \cite{V81} we parametrize the simple modules in $B$ combinatorially, as follows. Any involution $\iota$  in $W_n$ is encoded by its {\sl clan} $\sigma$, consisting by definition of a set of ordered pairs $(a_i,\epsilon_i) (a_i,b_i)^+$, and $(a_i,b_i)^-$, where every $\epsilon_i$ is a sign $+$ or $-$, every $a_i,b_i$ lies in $\{1,\ldots,n\}$), and every number from 1 to $n$ appears exactly once among the $a_i$ and $b_i$; given $\sigma$, the corresponding involution $\iota$ sends $a_i$ to $\epsilon_i a_i$ if $(a_i,\epsilon_i)$ lies in $\sigma$, flips $a_i$ and $b_i$ if $(a_i,b_i)^+$ lies in $\sigma$, and flips $a_i$ and $-b_i$ if $(a_i,b_i)^-$ lies in $\sigma$.  By convention we order the pairs in a clan by increasing order of their largest (or only) numerical coordinate and whenever a pair $(a_i,b_i)^+$ or $(a_i,b_i)^-$ occurs in $\sigma$ we assume that $a_i<b_i$.    Denote by $\mathcal S_n$ the set of all clans.  We call a clan even if the number of pairs $(a_i,b_i)^+$ appearing in it has the same parity as the number of pairs $(c_i,-)$ and odd otherwise; denote by $\mathcal S_n'$ the set of even clans and by $I_n'$ the corresponding set of involutions in $I_n$.  We take $\mathcal S_n'$ as a parametrizing set for the modules in $B$.  We will give a representation-theoretic interpretation of odd clans later, together with a parametrization of modules in the non-principal block $B'$ for $G$.

The Cartan involution $\theta$ corresponding to $\theta\in\mathcal S_n$ is constructed in the same way as for Sp$(p,q)$ \cite{M16}: it fixes the unit coordinate vector $e_i$ whenever $(i,\epsilon_i)\in\sigma$ with $\epsilon_i$ a sign; it interchanges $e_i$ and $e_j$ whenever $(i,j)^+\in\sigma$; and it interchanges $e_i$ and $-e_j$ whenever $(i,j)^-\in\sigma$ unless $j-i=\pm1$, in which case it sends both $e_i$ and $e_j$ to their negatives.  The root $e_{i+1}-e_i$ is compact imaginary whenever either $(i,\epsilon_i),(i+1,\epsilon_{i+1})\in\sigma$ with $\epsilon_i = \epsilon_{i+1}$ or $(i,i+1)^-\in\sigma$.  The root $e_2+e_1$ is compact imaginary whenever either $(1,\epsilon_1),(2,\epsilon_2)\in\sigma$ with $\epsilon_1=-\epsilon_2$ or $(1,2)^+\in\sigma$.

\section{The cross action and the Cayley transform}

\noindent For an element $\sigma$ of $\mathcal S_n$ and a pair $i,j$ of indices between 1 and $n$ we define In$(i,j,\sigma)$ as in \cite[Definition 1.9.1]{G93HC}, interchanging the unique occurrences of $i$ and $j$ in $\sigma$ and leaving all others (and all signs) unchanged, provided that at least one of the indices $i,j$ is paired with an index $k\ne i,j$ in $\sigma$; otherwise, we set In$(i,j,\sigma) = \sigma$.  

\begin{proposition}\label{proposition:crossaction}
Let $\alpha_i$ be the simple root $e_i-e_{i-1}, \sigma\in\mathcal S_n$.  Then $s_i\times\sigma$, the cross action of the simple reflection $s_i$ corresponding to $\alpha_i$ on $\sigma$, is given by In$(i-1,i,\sigma)$.  If $\alpha_1=e_2-e_1$, then the cross action of $s_1$ on $\sigma$ is given by changing the signs attached to (the singletons or pairs including) 1 and 2 in In$(1,2,\sigma)$.  The Cayley transform of $\sigma$ through $\alpha_i$ is obtained from $
\sigma$ by replacing the pairs $(i-1,\epsilon)(i,-\epsilon)$ in it by $(i,i+1)^+$ (if $\alpha_i$ is noncompact imaginary for $\sigma$ or by replacing $(i,i+1)^+$ by either $(i-1,+)(i,-)$ or $(i-1,-)(i,+)$, thereby obtaining two clans, if $\alpha_i$ is real.  The Cayley transform of $\sigma$  through $\alpha_1$ similarly either replaces $(1,\epsilon)(2,\epsilon)$ in $\sigma$ by $(1,2)^-$ or $(1,2)^-$ by either $(1,+)(2,+)$ or $(1,-)(2,-)$.
\end{proposition}

\begin{proof} This follows from a direct computation, along the lines of \cite[\S\S1.8,9,12]{G93HC}.  Note that both the Cayley transform and the cross action preserve evenness (or oddness) of a clan.
\end{proof}

\section{$\tau$-invariants and wall-crossing operators in rank 2}

\noindent The definition of the
 $\tau$-invariant for Sp$(p,q)$ carries over to $G$, except that $2e_1$ is no longer a simple root while $e_1+e_2$ is a new simple root.  As noted in \cite{MP16}, the criterion for $e_1+e_2$ to lie in the $\tau$-invariant of a domino tableau is that the $2$-domino is vertical and lies either in the leftmost column of the tableau or the column to its right.  We extend these definitions to $\mathcal S_M,\mathcal S_M'$ as in \cite{G93HC}.  The root $e_1+e_2$ lies in the $\tau$-invariant of an involution $\sigma$ if and only if either it is sent to a negative root by $\theta$ or the indices $1,2$ are paired with opposite signs $\epsilon_1,\epsilon_2$ in $\sigma$.  Given a pair $\{\alpha,\beta\}$ of distinct nonorthogonal simple roots, we define the wall-crossing operator $T_{\alpha\beta}$ on the level of clans (or simple Harish-Chandra modules) as in \cite{G93HC} and \cite{V79}.  It is single-valued, implemented either by the cross action by the simple reflection corresponding to $\alpha$ or $\beta$ or the Cayley transform through $\alpha$ or $\beta$ itself,, whichever of these operations has the required effect on the $\tau$-invariant of $\sigma$.  We will define these operators on domino tableaux below and show that their actions on clans and domino tableaux are compatible.  There is no uniform recipe for their action on signed tableaux (unlike the situation for Sp$(p,q)$), but their action on associated varieties of simple Harish-Chandra modules is always trivial.

\section{Wall-crossing operators in rank 4}

\noindent We also have wall-crossing operators $T_{D\beta}$ and $T_{\beta D}$ attached to the unique set $\{e_2-e_1,e_2+e_1,e_3-e_2,e_4-e_3\}$  of simple roots spanning a root system $R$ of type $D_4$, if the rank of $G$ is at least 4 \cite{GV92,MP16}.  These operators can be either single- or double-valued.  More precisely, the operator $T_{D\beta}$ takes clans of type $\mathcal A_\beta$ in the sense of \cite[4.4]{MP16} to clans of type $\mathcal D$; the operator $T_{\beta D}$ does the reverse. If the indices $1,2,3,4$ are all paired with each other or with signs in $\sigma$ and if $\sigma$ is of type $\mathcal A_\beta$, then $\sigma'=T_{D\beta}\sigma$ is obtained from $\sigma$ by replacing $(1,\epsilon)(3,\epsilon')$ in $\sigma$ by the single pair $(1,3)^{\epsilon\epsilon'}$, where we multiply the signs $\epsilon,\epsilon'$ in the obvious way.  The operator $T_{\beta D}$ reverses this action, replacing $(1,3)^{\epsilon}$ in $\sigma$ by $(1,\epsilon_1)(3,\epsilon_2)$, with the $\epsilon_i$ chosen in the two possible ways to make their product equal to $\epsilon$.  In general, the action of $T_{D\beta}$ or $T_{\beta D}$ on a clan $\sigma$ in its domain is more complicated to write down, but it is always implemented by a composition of cross actions and Cayley transforms corresponding to simple roots in $R$.  We will use this fact to show that the actions of these operators on clans and domino tableaux are compatible in \S9. These operators act trivially on associated varieties.

 \section{The algorithm}
 
\noindent As for Sp$(p,q)$ in \cite{M16} we now attach a pair $H(\sigma)=(\T_1,\oT_2)$ to a parameter $\sigma$, where $\T_1$ is a domino tableau and $\oT_2$ is an equivalence class of signed tableaux of the same shape as $\T_1$, this shape being a doubled partition of $2n$, that is, one whose parts occur in equal pairs.  Any representative $\T_2$ of $\oT_2$ will thus also have rows occurring in pairs of equal length; each row in a pair (called a double row) will begin with the same sign.   We define the equivalence relation on signed tableaux as in \cite{M16}, replacing even double rows by odd ones, so that any signed tableau is equivalent to any other obtained from it by changing all signs in any pair $D_1,D_2$ of double rows of the same odd length, of or different odd lengths whenever there is an open cycle with its hole in one of $D_1,D_2$ and its corner in the other.  The algorithm for computing $H(\sigma)$ is quite similar to its counterpart for Sp$(p,q)$ in \cite{M16}, but different enough that we give it in detail.
  
\begin{definition}\label{definition:addsingleton}
Let $\sigma\in\mathcal S_n$.  Order the pairs in $\sigma$ as above, by increasing size of their largest or only numerical coordinates.  Construct the pair $H(\sigma) = (\T_1,\oT_2)$ inductively, starting from a pair of empty tableaux.  At each step we insert the next element $(i,\epsilon)$ or $(i,j)^\epsilon$ into the current pair of tableaux.  Assume first that the next element of $\sigma$ is $(i,\epsilon)$ and choose any representative $\T_2$ of $\oT_2$.

\begin{tenumerate}

\item If the first double row of $\T_2$ ends in $-\epsilon$, then add $\epsilon$ to the ends of both of its rows and add a vertical domino labelled $i$ to the end of the first double row of $\T_1$.

\item If not and if the first double row of $\T_2$ is even, then we look first for a lower double row of $\T_2$ with the same length ending in $-\epsilon$; if these is such a double row, interchange it with the first double row of $\T_2$ and proceed as above.  Otherwise we start over, trying to insert $(i,\epsilon)$ into the highest double row of $\T_2$ strictly shorter than its first double row.  (In the end we may have to insert a domino labelled $i$ into a new double row of $\T_1$, using $\epsilon$ for both of the signs in the new double row of $\T_2$.)

\item If not and the first (or first available) double row of $\T_2$ has even length but there is more than one double row of this length, but none ending in $-\epsilon$, then we change all signs in the first two double rows of $\T_2$ of this length and then proceed as in the previous case.

\item Otherwise the highest available double row $R$ in $\T_2$ has odd length, ends in $\epsilon$, and is the only double row of this length.  In this case we look at the domino in $\T_1$ occupying the last square in the lower row of $R$.  If we move $\T_1$ through the open cycle of this domino, we find that its shape changes by removing this square and adding a square either at the end of the higher row of some double row $R'$ of $\T_1$ or else in a new row, not in $\T_1$.  It if lies in a new row, then change all signs in $R$ and proceed as above.  If it does not lie in a new row and $R'\ne R$, then change the signs of $\T_2$ in both $R$ and $R'$ and proceed as above (again not actually moving $\T_1$ through this open cycle).  Finally, if $R=R'$, then move $\T_1$ through the open cycle, place a new horizontal domino labelled $i$ at the end of the lower row of $R$ in $\T_1$ and choose the signs in $\T_2$ so that both rows of $R$ now end in $\epsilon$ while all other signs in $\T_2$ remain the same as before.
\end{tenumerate}
\end{definition}

\begin{definition}\label{definition:addpair}
Retain the notation of the previous definition but assume now that the next element of $\sigma$ is $(i,j)^\epsilon$.  We begin by inserting a horizontal domino labelled $i$ at the end of the first row of $\T_1$ if $\epsilon=+$, or a vertical domino labelled $i$ at the end of the first column of $\T_1$ if $\epsilon=-$, following the procedure of \cite{G90} (and thus bumping dominos with higher labels as needed).  We obtain a new tableau $\T'$ whose shape is obtained from that of $\T_1$ by adding a single domino $D$, lying either in some double row $R$ of $\T_1$ or else in a new row (in which case $D$ must be horizontal).  Let $\ell$ be the length of $R$ (before $D$ was added).  

\begin{tenumerate}

\item If $D$ is horizontal and $\ell$ is odd, then add a domino labelled $j$ to $\T'$ immediately below the position of $D$, in the lower row of $R$.  Choose signs in $\T_2$ so that both rows of $R$ end in the same sign as before, leaving all other signs the same. 

\item If $D$ is horizontal and $\ell$ is even, then $\T'$ does not have special shape but its shape becomes special if one moves through just one open cycle.  Move through this cycle, so that $R$ is now a genuine double row, and choose the signs in $\T_2$ so that both rows of $R$ end in different signs than they did before, leaving all other signs the same.  If a new double row has been created, put $+$ signs in both squares of $\T_2$.  Now $\T_2$ has either two more $+$ signs than before or else two more $-$ signs.  Insert a vertical domino labelled $j$ to the first available double row in $\T_1$ strictly below $R$, using the procedure of the previous definition.  The sign attached to $j$ is $-$ if $\T_2$ gained two $+$ signs and is $+$ otherwise.

\item If $D$ is vertical and $\ell$ is odd, then $R$ is still a double row; choose signs so that its rows now end in the same sign as they did before, leaving all other signs unchanged.  If a new double row was created, then put two $+$ signs (for definiteness) into the corresponding squares in $\T_2$.  Add a vertical domino labelled $j$ to the first available double row strictly below $R$ in $\T'$, as in the previous case, attaching to it this time the {\sl opposite} sign used in that case (so that if $\T_2$ gained two $+$ signs at the first step, then the sign attached to $j$ is again $+$, and similarly for $-$).

\item If $D$ is vertical and $\ell$ is even, then proceed as in the previous case.
\end{tenumerate}
\end{definition}
\vskip .1in

\noindent As for Sp$(p,q)$ one can check that the equivalence class $\oT_2$ does not change if $\T_2$ is replaced at any stage with another representative of its equivalence class.  To compute the associated variety of the Harish-Chandra module corresponding to $\sigma$ we choose any representative of $\oT_2$ and normalize it so that all odd double rows begin with $+$; as for Sp$(p,q)$, this variety is the closure of one nilpotent orbit \cite[5.2]{T07} and such orbits (via the Kostant-Sekiguchi bijection) are parametrized by signed tableaux with shape a doubled partition of $2n$ in which odd rows begin with $+$, except that such tableaux with only even rows correspond to two orbits \cite[9.3.4]{CM93}.

\section{Wall-crossing operators and domino tableaux}

\noindent We define the $\tau$-invariant of a pair $(\T_1,\oT_2)$ to be that of $\T_1$, as in \cite{M16}. The operator $T_{\alpha\beta}$ is defined on  $\T_1$ as in \cite{M16} if $\alpha,\beta$ are adjacent simple roots $e_i-e_{i-1},e_{i+1}-e_i$ or $e_{i+1}-e_i,e_i-e_{i-1}$ with $i\ge2$.  Thus in particular these operators $T_{\alpha\beta}$ always preserve tableau shape.  If $(\alpha,\beta)$ is $(e_3-e_2,e_2+e_1)$ or $(e_2+e_1,e_3-e_2)$ then $T_{\alpha\beta}$ is defined on a domino tableau as in \cite[4.3.5]{MP16} replacing the extended open cycle of the 3-domino in the first tableau relative to the second by the cycle including the 3-domino.  This cycle is always closed, since the shape of the domino tableaux is a doubled partition; if the cycle including the 3-domino were open, it would have to be  simultaneously up and down in the sense of \cite[\S3]{G93}, a contradiction.  We define the operators $T_{D\beta}$ and $T_{\beta D}$ on domino tableaux $\T_1$ as follows. Assume first that the first four dominos of $\T_1$ from a subtableau $S$ of special shape.   If the cycle of the 4-domino includes the 3-domino, then both operators send $\T_1$ to two tableaux, one obtained from $\T_1$ by interchanging the 3- and 4-dominos, the other obtained from this one by moving through either or both of the cycles through the 3- and 4-dominos, whichever of these is closed.  If the cycle of the 4-domino does not include the 3-domino, then the unique image of $\T_1$ is the tableau obtained from it by interchanging the 3- and 4-dominos.  If $S$ does not have special shape, then move through the appropriate cycles of $\T_1$ to make $S$ have special shape, as in \cite[\S4]{MP16}; these cycles are necessarily closed since if they were open they would necessarily be both up and down in the sense of \cite[\S3]{G93}, a contradiction.  Then proceed as in the case where $S$ is special; there is just one tableau in the image in this case.  It can happen that the image of a clan $\sigma$ under one of these operators consists of two clans $\sigma_1,\sigma_2$ while the image of the domino tableau $\T_1$ attached to $\sigma$ by $H$ is single-valued; if so then $H$ always attaches the same domino tableau to both $\sigma_i$.   As mentioned above, there is no uniform formula for the action of wall-crossing operators on signed tableaux, but the induced action on associated varieties is trivial.

\section{$H$ commutes with $\tau$-invariants}

\noindent As in \cite{G93HC} and \cite{M16}, we prove that our algorithm $H$ computes annihilators of simple Harish-Chandra modules by showing that it commutes with taking $\tau$-invariants and applying wall-crossing operators.  In this section we deal with $\tau$-invariants.  

\begin{proposition}\label{proposition:tau}
Let $\sigma\in\mathcal S_n$ and let $\alpha$ be a simple root for $G$.  Then $\alpha\in\tau(\sigma)$ if and only if $\alpha\in\tau(H(\sigma)$.
\end{proposition}

\begin{proof} This is proved in essentially the same way as \cite[Proposition 7.1]{M16}, omitting the case $\alpha=2e_1$, which does not arise in type $D$.  It is replaced by the case $\alpha=e_2+e_1$, which is again easily handled directly.
\end{proof}

\section{$H$ commutes with wall-crossing operators}

We complete our program of showing that the map $H$ computes annihilators by showing that it commutes with wall-crossing operators.

\begin{proposition}\label{proposition:wall-cross} Let $T$ be a wall-crossing operator $T_{\alpha\beta}, T_{D\beta}$, or $T_{\beta D}$.  Let $\sigma$ be a clan in the domain of $T$.  Then the clan or clans in $T\sigma$ have domino tableaux obtained by applying $T$ to the domino tableau $\T_1$ in $H(\sigma)$.
\end{proposition}

\begin{proof} If $T=T_{\alpha\beta}$ where $\alpha,\beta$ are adjacent simple roots with neither of these being $e_2+e_1$, then this follows as in the proof of \cite[Proposition 8.1]{M16}.  If one of these roots is $e_2+e_1$, then this follows by a direct calculation in rank 3 together with \cite[2.2.9]{G92}.  If $T=T_{D\beta}$ or $T_{\beta D}$ then it follows in all cases from the bijectivity of $H$ (proved below) that the clan or clans corresponding to the tableaux in $T\T_1$ are obtained from $\sigma$ by a composition of cross actions and Cayley transforms, so arise from $\sigma$ by the action of the Weyl group $W'$ of type $D_4$ corresponding to the root system $R$ used to define $T$ and lying in the Weyl group $W_n$ of $G$.  Since $H$ is already known to commute with wall-crossing operators of rank 2, the clan or clans satisfy the conditions to be of type $\mathcal D$ or $\mathcal A_\beta$ (according as $T=T_{D\beta}$ or $T=T_{\beta D}$).  But the only clan or clans satisfying these conditions and arising from $\sigma$ via the action of $W'$ are the value or values of $T\sigma$, whence the result follows.  (As mentioned above, it is possible for $T\T_1$ to have just one value while $T\sigma$ has two, but in that case the domino tableaux attached to the clans in $T\sigma$ are always the same.)
\end{proof}

\begin{theorem}\label{theorem:annihilators} Let $\sigma\in\mathcal S_n$.  Then the first coordinate $\T_1$ of $H(\sigma)$ parametrizes the annihilator of the simple Harish-Chandra module corresponding to $\sigma$ via the classification of \cite[Theorem 4.8.3]{MP16}.  The modules corresponding to $\sigma,\sigma'$ lie in the same cell whenever the second coordinates $\oT_2,\oT_2'$ of $H(\sigma),H(\sigma')$ agree up to changing the signs in rows of odd length.
\end{theorem}

Since the operators $T_{\alpha\beta},T_{D\beta},T_{\beta D}$ generate the cells of Harish-Chandra modules of trivial infinitesimal character for $G_0$, by \cite[Theorem 7]{M98} and \cite[Theorem 4.6.2]{MP16}, the clans $\sigma,\sigma'$ corresponding to two such modules in the same cell are such that any representatives $\T_2,\T_2'$ of the equivalence classes $\oT_2,\oT_2'$ attached to $\sigma,\sigma'$ by $H$ agree up to changing the signs in double rows of odd length.  

\section{$H$ is a bijection}

\begin{theorem}\label{theorem:bijection} The map $H$ defines a bijection between $\mathcal S_n$ and ordered pairs $(\T_1,\oT_2)$, where $\T_1$ is a domino tableau with shape a doubled parition of $2n$ and $\oT_2$ is an equivalence class of signed tableaux of the same shape as $\T_1$ such that the two rows in every double row begin with the same sign.
\end{theorem}

\begin{proof} We use the group $G'=O^*(2n)$ of matrices in $O(2n,\mathbb C)$ preserving the skew-Hermitian form used to define $G$.  This group has two components, with $G$ as its identity component.  Every simple Harish-Chandra module $Z$ for $G$ either already has the structure of a Harish-Chandra module for $G'$, in which case there is exactly one other simple Harish-Chandra module $Z'$ for $G'$ isomorphic to $Z$ as a module for $G$, or else there is a unique simple Harish-Chandra module $Z'$ for $G$ not isomorphic to $Z$ such that the direct sum of $Z$ and $Z'$ is a simple module for $G'$.  Now the surjectivity is proved in the same way as for \cite[Theorem 9.1]{MP16}, interchanging the  roles of double rows of even and odd length, and bearing in mind that two $+$ signs and two $-$ signs are added to the signed tableau in cases (1) and (2) of Definition 6.2, while four equal signs are added to this tableau in cases (3) and (4) of that definition.  Then injectivity follows since  by \cite{M98} there are exactly two cells of simple Harish-Chandra modules for $G'$ of trivial infinitesimal character for every nilpotent orbit in $\mathfrak g_0$ corresponding to a doubled partition with at least one odd part and one such cell corresponding to each of the two orbits with partition $\mathbf p$ for every doubled partition $\mathbf p$ of $2n$ with no odd parts.
\end{proof}

\begin{corollary}\label{corollary:bijection} Given an ordered pair $(\T_1,\oT_2)$ in the range of $H$ and a representative $\T_2$ of $\oT_2$, exactly one of $(\T_1,\oT_2)$ and $(\T_1,\oT_2')$ lies in the range of $H$ when restricted to $\mathcal S_n'$, where $\T_2'$ is obtained from $\T_2$ by changing the signs in just one odd double row, provided that $\T_2$ has at least one double row.  If all double rows of $\T_2$ are even, so that $\oT_2$ consists of $\T_2$ alone, then $(\T_1,\oT_2)$ lies in the range of $H$ restricted to $\mathcal S_n$ if and only if the number $n_v$ of vertical dominos in $\T_1$ is congruent to 0 mod 4.
\end{corollary}

\begin{proof} This follows from the module structure of the cells in $G$ described in \cite[Theorem 10]{M98}, together with the criterion in \cite[Corollary 4.1.4]{MP16} for an ordered pair $(\T_1,\T_2)$ of domino tableaux of the same doubled partition shape with only even rows to correspond to an element of the Weyl group $W$ of type $D_n$ in the bijection of \cite{G90}.  
\end{proof}

Given now an ordered pair $(\T_1,\oT_2)$ in the range of $H$ with $\T_1$ having at least odd double row and tableaux $\T_2,\T_2'$ as in the statement of the corollary, exactly one of the clans $\sigma,\sigma'$ corresponding to the ordered pairs $(\T_1,\oT_2),(\T_1,\oT_2')$ is even and so corresponds to a simple Harish-Chandra module $Z$ for $G$.  We regard the other clan as parametrizing the other simple Harish-Chandra module for $G'$ with the same restriction to $G$ as $Z$.  If on the other hand $\T_1$ has only even rows and $n_v$ is a multiple of 4, then the clan $\sigma$ corresponding to $(\T_1,\oT_2)$ corresponds to a module $Z$ in the principal block $B$ for $G$. If $n_v$ is not a multiple of 4, then $Z$ lies in the nonprincipal block $B'$ of $G$.  Such modules $Z$ pair up with unique ones $Z'$ in the principal block in such a way that the direct sum of $Z$ and $Z'$ is a simple Harish-Chandra module for $G'$.

It remains to show that if $\sigma\in\mathcal S_n, H(\sigma) = (\T_1,\oT_2)$, and if $\T_2$ is a representative of $\oT_2$, then the module $Z$ corresponding to $\sigma$ has associated variety the closure of the $K$-orbit in $\mathfrak p$ corresponding to a (suitable normalization of) $\T_2$ via the Kostant-Sekiguchi correspondence.  This we do in the next and final section.

\section{Associated varieties} 

We recall that nilpotent $G'$-orbits in $\mathfrak g_0$ are parametrized by signed tableaux with shape a doubled partition of $2n$ such that odd rows begin with $+$.  Nilpotent $G$-orbits in $\mathfrak g_0$ are parametrized in the same way, except that the single $G'$-orbit corresponding to a signed tableau with only even rows splits into two $G$-orbits.  Given a clan $\sigma\in\mathcal S_n$ with $H(\sigma) = (\T_1,\oT_2)$ and a representative $\T_2$ of $\oT_2$, let $\T$ be obtained from $\T_2$ by changing signs as necessary to make its odd rows begin with $+$.

\begin{theorem}\label{theorem:associated varieties} Let $\sigma\in\mathcal S_n', H(\sigma)=(\T_1,\oT_2)$, and define $\T$ from $\oT_2$ as above.  The module $Z$ for $G$ corresponding to $\sigma$ has associated variety the unique one parametrized by $\T$ if $\T$ has at least one odd row, or one of the two parametrized by $\T$ if $\T$ has only even rows (the same orbit for all clans $\sigma$ with the same $\T$). 
 \end{theorem}

\begin{proof}
The proof is similar to that of \cite[Theorem 10.1]{M16}.  Let $\mathfrak q$ be a $\theta$-stable parabolic subalgebra of $\mathfrak g_0$ whose corresponding Levi subgroup $L$ of $G'$ is $O^*(2m)\times U(p_1,q_1)\times\cdots\times U(p_r,q_r)$, where $2m+p_1+q_1+\ldots+p_r+q_r = 2n$.   There is an irreducible derived functor module $A_{\mathfrak q}$ of trivial infinitesimal character whose associated variety corresponds to the Richardson orbit $\mathcal O$ attached to $\mathfrak q$ in the sense of \cite{T05}.  Its clan $\sigma'$ has one block of terms for every simple factor in $L$.  If $m\ge1$, the first block of terms can be either $(1,2)^-\cdots(m-1,m)^-$ or $(1,2)^+(3,4)^-\cdots(m-1,m)^-$, if $m$ is even, or $1^+(2,3)^-\cdots(m-1,m)^-$ or $1^-(2,3)^-\cdots(m-1,m)^-$, if $m$ is odd.  The blocks of terms corresponding to the $U(p_i,q_i)$ factors are defined as in the proof of \cite[Theorem 10.1]{M16}, with the roles of initial $+$ and  $-$ signs reversed, so that factor $U(p,q)$ with $p>q$ corresponds to a block of terms beginning with $p-q$ singleton indices with $-$ signs attached rather than $+$ ones.  Letting $H(\sigma') = (\T_1,\oT_2)$ and defining $\T$ as above from $\oT_2$, one checks immediately that the orbit corresponding to $\T$ is indeed $\mathcal O$.  Since both $\T$ and associated varieties are constant across cells of Harish-Chandra modules, the desired result holds whenever $\T$ corresponds to a Richardson orbit.  

Given $Z$ as in the theorem, let $\bar{\mathcal O}$ be its associated variety, transferred via Kostant-Sekiguchi to a real nilpotent orbit.  Using induction by stages, we can induce $\mathcal O$ to an orbit $\mathcal O'$ for a higher rank group $G'$ such that all odd parts in the partition of $\mathcal O'$ have multiplicity at most 4, whence $\mathcal O'$ is Richardson by \cite[Corollary 6.2]{T05}.  The theorem then holds for the module $Z'$ correspondingly induced from $Z$.  But now $\mathcal O$ is the only orbit of $G$ inducing to $\mathcal O'$ relative to a suitable parabolic subalgebra $\mathfrak q$ of Lie~$G'$ having $G$ has the only simple factor of type $D$ in its Levi subgroup.  The theorem follows. 
\end{proof} 

For $G'$ the result is the same, except that the associated variety of a simple module $Z$ which is reducible over $G$ is the closure of the single $G'$-orbit (or the union of the closures of the two $G$-orbits) corresponding to $\T$; note that $\T$ has only even rows in this case.

\end{document}